\documentclass[12pt]{article}
\usepackage{amsmath,amsthm,amsfonts,amssymb,times,hyperref}
\usepackage{xcolor}

\newtheorem{theorem}{Theorem}
\newtheorem{proposition}{Proposition}
\newtheorem{lemma}{Lemma}
\newtheorem{corollary}{Corollary}

\newtheorem{example}{Example}

\theoremstyle{definition}
\newtheorem{definition}{Definition}

\newcommand{\vctr}[1]{\vec{#1}}

\allowdisplaybreaks

\title{Irreducibility of random lacunary 0,1-polynomials}
\author{
Alexandros Kalogirou \\
Erd\H{o}s Center \\
R\'enyi Institute \\
Budapest, H-1053 \\
Email: kalogira@email.sc.edu}

\newcommand{\Q}{\mathbb{Q}}
\newcommand{\Z}{\mathbb{Z}}
\newcommand{\Po}{\mathbb{P}}
\begin{document}
\maketitle
\begin{abstract}
We show that $0,1$-polynomials of high degree and few terms are irreducible with high probability. Formally, let $k\in\mathbb{N}$ and $F(x)=1+\sum_{i=1}^kx^{n_i}$, where $ 0<n_1<\cdots<n_k\leq N. $ Our main Theorem 1 establishes $\lim_{k\rightarrow\infty}\limsup_{N\rightarrow\infty}\Po(\text{$F(x)$ has a cyclotomic factor})=0.$ Then we show in Theorem 2 that $\lim_{k\rightarrow\infty}\limsup_{N\rightarrow\infty}\Po(\text{$F(x)$ is reducible})=0$, using a classical result of Schinzel. The probability in this context is derived from the uniform count of polynomials $F(x)$ of the above form. Throughout the text, factorization is considered within the ring $\Q[x].$
\end{abstract}
\section{Introduction}
Irreducible polynomials in $\mathbb{Z}[x]$ are abundant. In 1936, B. L. van der Waerden \cite{Waerden}, proved that polynomials of the form $f(x)=x^k+\sum_{i=0}^{k-1}a_ix^i,$ with fixed $k$ where each $a_i$ is chosen uniformly and independently from $[-H,H]\cap\Z,$ are irreducible with probability tending to 1, as $H$ grows to infinity. This is the so called \textit{large box model}. The state of the art results in this set up, where also the task of determining the exceptions to the general rule that $\text{Gal}_f=S_k$ is undertaken, are given by M. Bhargava \cite{barghava}.

Progress has been slower when fixing the height of the polynomial family and letting the degree grow. We refer to this as the \textit{restricted coefficients model}. A major motivating conjecture in this setting, was suggested by A. M. Odlyzko and B. Poonen \cite{OP}, where they conjecture that almost all $0,1$-polynomials with constant term 1, are irreducible, by letting the degree grow.

In way of progress in the conjecture, S. Konyagin (1999) \cite{Sergei} showed that at least $c2^k/\log k$ of $0,1$-polynomials of degree $k$ are irreducible. Konyagin's work was used crucially in \cite{Bary}\cite{bv}\cite{BKK}. E. Breuillard and P. P. Varj\'u in 2019 \cite{bv}, resolved the Odlyzko-Poonen conjecture positively under the Extended Riemann Hypothesis. On a related note, in 2020 L. Bary-Soroker and G. Kozma \cite{Bary} showed that if all the coefficients $a_i\in[1,210]$, then as the degree $k$ tends to infinity, polynomials as above are irreducible with probability tending to $1$. This was improved for more general distributions and with several strengthenings of the conclusion in follow up work by the same authors and D. Koukoulopoulos \cite{BKK}, showing among other results that the $a_i$'s can be chosen uniformly from the interval $[1,35]$. They also show that a positive proportion of $0,1$-polynomials are irreducible unconditionally.  More recently, D. Hokken and D. Koukoulopoulos in a preprint \cite{hkrecip} have adapted the approach first introduced in \cite{Bary} to the extremely interesting case of reciprocal polynomials. Their work also provides an excellent list of references to recent work that extend the ones we mention here, as well as adapt the methods to different settings. To name a few, H. T. Pham and M. W. Xu \cite{depend} consider the large box model with relaxed conditions on the distribution of the coefficients, notably introducing dependencies and non-uniformity, while the work of S. Eberhard \cite{eberhard} applies the techniques of the restricted coefficients model to the study of the characteristic polynomials of random matrices.

We are hereby considering a twist to the above problem, by focusing instead on lacunary $0,1$-polynomials. It is worth pointing out that all three variants mentioned, meaning the large box model, the restricted coefficients model, and our lacunary setting, use very different methods, we therefore give the lacunary setting the label of \textit{random exponents model}. A. Schinzel contributed very impactful research on lacunary polynomials with respect to divisibility, and his contribution to our work is perhaps unfairly reduced to the claim of Proposition 1 given in Section 2. In M. Filaseta and A. Kalogirou \cite{integers}, a friendly exposition of a similar result, restricted to the 0,1-polynomial case, is attempted, which also draws lines of comparison with Schinzel's arguments. 

\begin{definition}\label{Fdefinition}
    Sample uniformly at random a subset of $\{1,2,\dots,N \}$ consisting of $k$ elements, and label these elements $0<n_1<\cdots<n_k\leq N$. Consider the random polynomial 
    \begin{equation}\label{definition}
    F(x)=1+\sum_{i=1}^kx^{n_i}
    \end{equation}
    When using probability notation in the sequel, say $\Po(F(x)\text{ satisfies }A)$, we mean the proportion of the total ${N \choose k}$ such $F(x)$ that satisfy property $A$.
\end{definition}

Our main result is the following.
\begin{theorem}
For $F(x)$ as above, we have
\begin{equation}
\lim_{k\rightarrow\infty}\limsup_{N\rightarrow\infty}\Po(\text{$F(x)$ has a cyclotomic factor})=0.
\end{equation}

\end{theorem}
Combining the above with Proposition 1, we arrive at the following.
\begin{theorem}
For $F(x)$ as in Theorem 1, we have
\begin{equation}
\lim_{k\rightarrow\infty}\limsup_{N\rightarrow\infty}\Po(\text{$F(x)$ is reducible})=0.
\end{equation}

\end{theorem}

A couple of comments are in order. First, the two successive limits in Theorem 1 are essential. One cannot expect an almost-all result with $k$ fixed and $N\rightarrow\infty$. Indeed, for fixed $k$, a positive proportion of $F$ as in \ref{Fdefinition} will be divisible by $(x^{k+1}-1)/(x-1)$. In fact, as explained in the next section giving the outline of the proof, in critical ranges for $n$ roughly up to the size of $k$, a rather precise count of polynomials $F(x)$ divisible by $\Phi_n(x)$ is attained, as a means to proving Theorem 1.

Second, an extra resource, unique to the lacunary setting, is needed when dealing with cyclotomic factors $\Phi_n(x)$ of very high degree $n$. This comes in the form of strong finiteness results of J. H. Conway and A. J. Jones \cite{conjones} and M. Filaseta and A. Schinzel \cite{filsch}. This is the topic of the last section.

Very recently, Kogler \cite{kogler} proved the Odlyzko–Poonen conjecture unconditionally. His argument carries out the program outlined in \cite{integers} and provides an interesting example of ideas originating in the study of lacunary polynomials being applied to the restricted-coefficient model

\section{Proof outline}

The proof of Theorem 2 rests on the following proposition drawn from \cite[Corollary 2]{schinzel}, as well as Theorem 1.

\begin{proposition}
Let $\vctr{\alpha}=[\alpha_0,\dots,\alpha_k]\in\mathbb{Q}^{*k+1}$. The number of integer vectors $\vctr{n}=[n_1,\dots,n_k]$ such that $0<n_1<\dots<n_k\leq N$ and $\alpha_0+\sum_{j=1}^k\alpha_jx^{n_j}  $ has a non-trivial, non-cyclotomic factor, is less than $c(\vctr{\alpha})N^{(k+1)/2}$, where $c(\vctr{a})$ is a positive constant depending on $\vctr{a}$ only.
\end{proposition}

Here, we are using the above fact for $\alpha_0=\dots=\alpha_k=1$. 

It follows from Theorem 1 that most $F(x)$ for big enough $k$ will not have cyclotomic factors, and therefore by Proposition 1, when restricting to those $F(x),$ outside a vanishing proportion as $\lim_{N\rightarrow\infty}N^{(k+1)/2}/{N\choose k}=0$, they must also be irreducible.

The proof of Theorem 1 will be achieved by bounding each individual $\Po(\Phi_n\mid F(x))$ and showing that their sum tends to zero, as $N$ and $k$ tend to infinity in that order. We split the bounds into three ranges in order to achieve this.

Throughout we consider the remainder $f(x)=F(x)\mod (x^n-1)$. The distribution of its coefficients is considered in the section Preliminaries I.

Strictly speaking the argument used for the second range can also be used to cover the first range, but the arguments given for the first range are of independent interest. The reasons are first that they give the underlying strategy for the entire paper, which is to bound the probability that a random sum of roots of unity, coming from $F(x)$, vanishes, or what is equivalent, lies in the lattice of relations between the roots of unity of given order. That is the topic of the section Preliminaries II, where an explicit description of the lattice structure of the relations in the roots of unity is given, in particular we describe a basis. The second reason is that, up to perhaps logarithmic factors, the estimates we give for the first range can be turned into exact asymptotics.

In the second range, we bound the probability $\Po(\Phi_n\mid F(x))=\Po(\Phi_n\mid f(x))$, by splitting $f(x)=g(x)+h(x),$ where $g(x)$ is sum of the monomials of $f(x)$ up to degree $\phi(n)-1,$ and $h(x)$ consists of the monomials of higher degree. Then the condition $\Phi_n\mid f(x)$ implies that for $f(x)$ with given $h(x)$, there is only one possibility for $g(x),$ namely $g(x)=-h(x)\mod \Phi_n$. We therefore control $\Po(\Phi_n\mid f(x))$ by the weight of the biggest atom in the conditional distribution $g(x)|_{h(x)}$. We give a uniform bound for the atoms of the distribution outside a range for $h(x)$ that is of measure approaching zero.

The last section deals with the very interesting range where $e^{k/(24\log k)}\leq n\leq N\log N$. The upper bound comes from the restriction $\deg \Phi_n\leq N$. This is in a way the heart of the paper as it deals with the range that is special to the case of lacunary polynomials. Unlike say the restricted coefficients set-up, where the degree of the possible divisors of the random polynomial $F(x)$ is controlled by the bound on the degree of $F(x)$, here the `real' parameter is $k$ and the degree of divisors is unbounded on $k$. The strategy for obtaining bounds from the previous two ranges cannot be extended because it is achieved by conditioning on the complement of some events, for example on the number of terms of $h(x)$ in the previous paragraph, that do not seem to have bounded probability as $N\rightarrow \infty$. That is what forces the introduction of some stronger results on vanishing sums of roots of unity. These are explored also in the section Preliminaries II. They are all to be traced in the work of J. H. Conway and A. J. Jones (1976) \cite{conjones}, while adjusted to the polynomial divisibility frame, they assume the form of Corollary 1 from M. Filaseta and A. Schinzel \cite{filsch}.

\section{Preliminaries I}
Here we discuss the distributions that will be used in the arguments and provide their justification.

The path to proving Theorem 1 begins with bounding the probability that $F(x)$ has a cyclotomic factor $\Phi_n(x)$, for fixed $k$ and $n$, and for $N$ growing to infinity. Detecting $\Phi_n(x)\mid F(x)$, will involve as a first step throughout this work, considering the remainder
$$  f(x)=F(x) {\hskip -5pt}\mod (x^n-1)=\sum_{j=0}^{n-1}c'_j x^j,  $$
where $c'_j$ counts the number of monomials in $F(x)$ whose exponents are congruent to $j$ modulo $n$. Then picking $\zeta_n$ any primitive $n$-th root of unity, we have $\Phi_n(x)\mid F(x)$ if and only if $f(\zeta_n)=0$. Instrumental to counting the latter probability will be to understand how the $n$-th dimensional vector $[c'_0,\dots,c'_{n-1}]$ is distributed, which we tend to now.

A small nuisance is that with our normalization of having constant coefficient 1 for $F(x),$ we immediately have $c'_0\geq 1$, and it will be convenient to refer to the following definition.
\begin{definition} We call
\[
\vctr{c}=[c_0,c_1,\dots,c_{n-1}]
=[c'_0-1,c'_1,\dots,c'_{n-1}],
\]
\textit{the coefficient vector} for $f(x)$.
\end{definition}
It should be clear that as $F(x)$ ranges, the distribution of $\vctr{c}$ is close to the multinomial distribution with $k$ trials and $n$ equally likely outcomes, at least for $n$ and $k$ small compared to $N$. We quantify the closeness of the actual distribution of $\vctr{c}$ with the above multinomial, and justify as a result, that it is acceptable in part of our estimates in the sequel, to use the multinomial distribution.

\begin{lemma}
    Let $\Po(\{\vctr{c}\})$ denote the probability that the coefficient vector of $f(x)$ is equal to a given $\vctr{c}$. Then
    \begin{equation}\label{multinomialest}
        \Po(\vctr{c})=\frac{1}{n^k}\binom{k}{c_0,c_1,\dots,c_{n-1}}\cdot \bigg(1 + O_{k,n}\bigg(\frac{1}{N} \bigg)\bigg).
    \end{equation}
    We are here assuming that $k,n$ are fixed, while $N\rightarrow\infty$. Notice that he main factor above is exactly the weight assigned to the coefficient vector $\vctr{c}$ by the multinomial distribution.
\end{lemma}
\begin{proof}
The number of polynomials $F(x)$ is equal to $\binom{N}{k}.$ Call $A_j$ the set of integers between $1$ and $N$ that are congruent to $j$ modulo $n$. Then the number of ways to achieve the coefficient vector $\vctr{c}$ is equal to
$$  \prod_{j=0}^{n-1}\binom{\#A_j}{c_j}.  $$ 
That is because the factors $\binom{\#A_j}{c_j}$ count the number of ways that we can pick $c_j$ elements from $A_j$, and once these are picked, the whole sequence of exponents $n_i$ of $F(x)$ can be recovered, through the ordering of the elements in all $A_j$'s. We get then that
\begin{align*}
\Po(\{\vctr{c}\})&=\frac{\prod_{j=0}^{n-1}\binom{\#A_j}{c_j}}{\binom{N}{k}}\\[5pt]
&=\frac{\prod_{j=0}^{n-1}(N/n)^{c_j} \big(1 + O_{k,n}\big(\frac{1}{N}\big)\big)^{c_j}/c_j!   }{N^k \big(1 + O_{k}\big(\frac{1}{N}\big) \big)^k/k! }\\[5pt]
&=\frac{1}{n^k}\binom{k}{c_0,c_1,\dots,c_{n-1}}\cdot \bigg(1 + O_{k,n}\bigg(\frac{1}{N} \bigg)\bigg),
\end{align*}

where the big-oh notation has its usual meaning so that $O_{k,n}(1/N)$ refers to a function which is bounded by a constant, only depending on $k$ and $n$, times $1/N$, for $N$ large. 
\end{proof}

As will become clear from what follows, the arguments that require any assumption about the distribution of $\vctr{c}$ are referring to a range of $n$ which is bounded by a function of $k$. With $k$ fixed and $N$ tending to infinity then, the error term, corresponding to the big-oh expression for $\Po(\{\vctr{c}\})$ above, tends to zero. Having proved the desired claim at the level of each atom $\vctr{c}$, we will use the convergence of $\Po(\{\vctr{c}\})$ to the multinomial distribution for events that might involve a number of different atoms $\vctr{c}$, utilizing that the number of these atoms is a bounded function of $k$.

Finally, we will use some concentration bounds for the multinomial distribution, all of which will be derived from the following Chernoff bound for the binomial distribution.
\begin{lemma}[Chernoff's bound]\label{chernoff}
Let $X\sim \text{Bin}(n,p)$, so $X$ satisfies a binomial distribution with $n$ possible outcomes each occurring with probability $p \in [0,1]$. Let $\mu=\mathbb{E}[X]$, the expected value of $X$. Then for any $\delta \in (0,1)$, we have
$$  \Po(|X-\mu|\geq \delta \mu)\leq 2\exp\left(\frac{-\delta^2\mu}{3}\right).  $$ 
\end{lemma}
\begin{proof}
    This is classic and one of many references, also available online, is Corollary 2.3 from \cite{randomgraphs}.
\end{proof}
\section{Preliminaries II}
Here we discuss some useful results about vanishing sums of roots of unity. They are drawn from J. H. Conway and A. J. Jones \cite{conjones}, but we repeat part of their treatment here to highlight how the results will be applied.

We introduce some of the necessary formalism. Consider the free $\Z$-module with basis the roots of unity. There is the obvious homomorphism from this module to the complex numbers, which is the evaluation map. The kernel of the evaluation map are the \textit{vanishing sums}.
\begin{definition}
    For a natural number $n$, we will call the product of its distinct prime factors as the \textit{squarefree part} of $n$, or its \textit{radical},. In the sequel, this will usually be denoted by $a(n)$ or $a$ where there is no danger of confusion.
\end{definition}

This is now Theorem 1 of \cite{conjones}. 
\begin{proposition}\label{Conway}
Let $S$ be a finite vanishing sum of powers of $\omega$, where $\omega$ is a primitive $n$-th root of unity. Write
\begin{equation}
    S=\sum_{i=0}^{b-1}S_i,
\end{equation}
where $S_i$ ranges over the terms in $S$ where the exponent of $\omega$ is congruent to $i\mod b$. Here $b=n/a(n)$, where $a(n)$ denotes the squarefree part of $n$.
We then have $S_i=0$, for $0\leq i\leq b-1$.
\end{proposition}

Pursuing further this line of argumentation, M. Filaseta and A. Schinzel in \cite{filsch} Corollary 1, show the following, which shall also be invoked in the sequel.
\begin{corollary}
Let $f(x)\in\Z[x]$ have $k+1$ non-zero terms. If $f(x)$ is divisible by a cyclotomic polynomial, then there is a positive integer $m$ such that
$$ 2+\sum_{p\mid m} (p-2)\leq k+1\text{     and     }\Phi_m\mid f(x).   $$
\end{corollary}
This will be used primarily to restrict the size of the squarefree part of $n$ for a candidate divisor $\Phi_n(x)$ of $F(x)$, in conjunction with the aforementioned decomposition into smaller vanishing sums in the spirit of Conway and Jones.

We now proceed with giving a rough bound on the squarefree part of a number $m$ satisfying the inequality of the Corollary. Since we are not interested in a tight bound, we proceed by successively relaxing the condition a bit.

\begin{lemma}
    Let $m$ as in Corollary 1 satisfy 
    \begin{equation}\label{conway}
        2+\sum_{p\mid m}(p-2)\leq k+1.
    \end{equation}
    Then $m$ has squarefree part $b(m)$ satisfying
    \begin{equation}\label{radicalbound}
        b(m)<30\cdot2^{2\sqrt{k}}\lceil(2\sqrt{k}+8)\rceil!.
    \end{equation}
\end{lemma}
\begin{proof}
The bound on $b(m)$ can only worsen if we assume that 2,3, and 5 divide $m$ or $30\mid m$, and at the same time consider the worse than \ref{conway} bound
$$   \sum_{p\mid m,\,p\geq 7}(p-2)\leq k.  $$
We want then to bound $30*\prod p$, where $p\geq 7$ are prime and satisfy the above inequality. We can further remove the condition of being prime, and instead just ask that the summands, call them $v_i$'s now, with $1\leq i\leq I$, are an increasing sequence of integers $\geq 7$, which again simplifying, satisfy
\begin{equation}\label{badbound}
\sum_i v_i\leq 7k/5.
\end{equation}
Generating the sequence $C_l=\sum_{j=1}^l(7+(j-1))$, by induction, we will deduce a bound on $\prod_i v_i$, with
$$ \sum_i v_i\leq t $$
for $C_l< t\leq C_{l+1}$. Let
$$  B_l=\max\{ \prod_i v_i: C_l<\sum_{i}v_i\leq C_{l+1}  \} $$
be the uniform bound for $\prod_i v_i$ for $t$ in the range $C_l<t\leq C_{l+1}$. When $0=C_0< t\leq C_1=7$, one easily verifies that $B_0=7$.

We prove by induction that for $l\geq 1$, we have $B_l\leq 2^{l}\prod_{j=1}^{l+1}(7+(j-1))$. Subtracting the biggest term $v_I$, we fall into a previous interval, say $C_{l'}<\sum_{i=1}^Iv_i-v_I\leq C_{l'+1}$, where $l'<l$. That is because $v_I>7+l$ for an increasing sequence $7\leq v_1<v_2<\cdots<v_I$ satisfying $\sum_{i=1}^Iv_i>C_l.$ We then have
\begin{align*}
\prod_{1\leq i \leq I} v_i&=v_I\prod_{1\leq i\leq I-1}v_i\\
&\leq (C_{l+1}-C_{l'}) \cdot B_{l'}\\
&=\left(\sum_{j=l'+1}^{l+1} (7+(j-1))\right)2^{l'}\prod_{j=1}^{l'+1}(7+(j-1))\\
&\leq 2\left(\sum_{j=l'+2}^{l+1} (7+(j-1))\right)2^{l'}\prod_{j=1}^{l'+1}(7+(j-1))\\
&\leq 2^{l'+1} \prod_{j=l'+2}^{l+1}(7+(j-1))\prod_{j=1}^{l'+1}(7+(j-1))\\
&\leq 2^{l}\prod_{j=1}^{l+1}(7+(j-1)).
\end{align*}
We notice that $2^l\prod_{j=1}^{l+1}(7+(j-1))$ is an increasing function of $l$, so choosing the biggest possible value for $l$ allowed by \ref{badbound} and $\frac{l^2}{2}\leq C_l$, we get that $l< 2\sqrt{k}.$ We conclude that
$$  b(m)\leq 30 \prod_{i}v_i\leq 30\cdot 2^{2\sqrt{k}}\prod_{j=1}^{2\sqrt{k}+1}(7+(j-1)) $$

Putting everything together with Corollary 1, that if $F(x)$ has a cyclotomic factor, then it is divisible by $\Phi_m(x)$, where $m$ has squarefree part at most
$$ 30\cdot2^{2\sqrt{k}}\lceil(2\sqrt{k}+8)\rceil! . $$
\end{proof}
Concluding this section, we give here a description of the space of vanishing sums involving $n$-th roots of unity. First, we introduce the $\Z$-module $M=\oplus_{i=0}^{n-1}\Z\cdot \zeta_n^i$, where $\zeta_n$ is a fixed primitive $n$-th root of unity. Like we described earlier, we can define the evaluation map from $M$ to $\mathbb{C}$, whose kernel are the vanishing sums or \textit{relations} in the roots of unity. Call this space of relations $R$. We then have, since the image of the evaluation map is $\Z[\zeta_n]$, which is of dimension $\phi(n)$, that $R$ has dimension $n-\phi(n)$.

Our main goal here is to describe a basis for $R$ as a $\Z$-module. The most convenient way to do this is by relabeling the basis elements of $M$. 

Let then $n=p_1^{e_1}\cdots p_j^{e_j}$ be the prime factorization of $n$, and call $\zeta_1,\dots, \zeta_j$ some fixed $p_1^{e_1},\dots,p_j^{e_j}$-th primitive roots of unity. No problem of confusion with $\zeta_n$ shall arise from this convention, so we suppress the appearance of the primes at hand from the notation. By the CRT, the module $M\cong \Z[\Z/n\Z]$ is isomorphic to $\Z[ \oplus_{i=1}^j\Z/p_i^{e_i}\Z ]$, and after writing
$$  \zeta_n^l=\prod_{i=1}^j \zeta_i^{a_i},  $$ 
we map the $l$-th coordinate of the module $M$ into the $j$-tuple $[a_1,\dots,a_j]$. We proceed in the description of a basis for $R$ by induction on the number of prime factors of $n$. When $n=p_1^{e_1}$, the basis is
$$B_1=\{\zeta_1^i(1+\zeta_1^{p_1^{(e_1-1)}}+ \zeta_1^{p_1^{(e_1-1)}\cdot 2}+\cdots+\zeta_1^{p_1^{(e_1-1)}(p_1-1)}),0\leq i\leq p_1^{e_1-1}-1\}. $$
A justification for this is as follows. First, observe that $\# B_1=p_1^{e_1}-\phi(p_1^{e_1})$, as should be. Also, notice that the elements of $B_1$ are orthogonal in a strong sense, that is their coordinates have disjoint support in the chosen basis for $M$. Finally, it is easy to check that quotienting $M$ by the elements of $B$ gives exactly the ring of integers $\Z[\zeta_1]$.

We proceed with a description of the basis for $j>1$. Having constructed the basis $B_{j-1}$, the basis $B_j$ can be taken to be the union of two sets, one of which is
$$ A=\{\zeta_j^i\cdot y, 0\leq i\leq p_j^{e_j}-1, y\in B_{j-1}\}.  $$
For the other set, consider any integral basis $B'$ of $\Z[\zeta_1\cdots \zeta_{j-1}]$ comprising of roots of unity, as well as the set
$$ B''=\{\zeta_j^i(1+\zeta_j^{p_j^{(e_j-1)}}+ \zeta_j^{p_j^{(e_j-1)}\cdot 2}+\cdots+\zeta_j^{p_j^{(e_j-1)}(p_j-1)}),0\leq i\leq p_j^{e_j-1}-1\} , $$
same as $B_1$, but with $p_1$ replaced with $p_j$. Now take 
$$ A'=\{xy,x\in B',y\in B''\}.   $$
We have 
$$ B_j=A\cup A'.  $$
\begin{lemma}
    The set $B_j$ defined above, is a basis for the evaluation kernel $R$.
\end{lemma}
\begin{proof}
We have that each ring $\Z[\zeta_i]$ for $1\leq i\leq j$ has $\{\zeta_i^r,0\leq r\leq \phi(p_i^{e_i})-1 \}$ as an additive basis, while at the same time
\begin{equation}\label{algebra}
    \Z[\zeta_n]=\otimes_{i=1}^j \Z[\zeta_i].
\end{equation}
See for instance \cite{Neukirch} for reference. It is obvious that $\langle B_j\rangle \subset R$. Our goal is then to show that $R\subset\langle B_j\rangle$. Consider the sequence of projections 
$$ \pi:M\rightarrow M/\langle B_j\rangle \rightarrow M/R=\otimes_{i=1}^j\Z[\zeta_i]\subset\mathbb{C}. $$
We conclude by showing that the second projection is one to one, therefore an isomorphism. Indeed, that is dictated by the choice for $B_j$. An arbitrary element 
$$m=\sum_{g=[g_1,\dots,g_j]\in \oplus_{i=1}^j\Z/p_i^{e_i}\Z } m_g\prod_{i=1}^j \zeta_i^{g_i},\quad m_g\in\Z, $$
of $M$, can be written when reduced $\mod \langle A\rangle$ as
$$ m=\sum_{(b',g_j)\in (B',\Z/p_j^{e_j}\Z)} m_{b',g_j} b'*\zeta_j^{g_j} , $$
while one more reduction $\mod A'$, gives the representation as
\begin{equation}\label{finalform} m=\sum_{(b',d)\in(B',[0,\phi(p_j^{e_j})-1])} m_{b',d}b'*\zeta_j^d, \end{equation}
in $M/\langle B_j\rangle$. Given \eqref{algebra}, we see that a non-zero element $m+\langle B_j\rangle$, necessarily having at least one coefficient $m_{b',d}$ non-zero, evaluates non trivially inside $\mathbb{C}$. The upshot is that $\langle B_j\rangle =R$. Combining this with the easy to check $\# B_j=n-\phi(n)$, we conclude that $B_j$ is indeed a basis for $R$.
\end{proof}

We offer here a small example of the above construction.
\begin{example}
We present a basis for the space of relations among powers of $\zeta_{12}$. We have $12=2^2\cdot3$. Also $\phi(12)=4$, so our basis should have 8 elements. Call $i$ the fourth root of unity and $\omega$ a fixed third root of unity. We first consider 

$$ B_1=\{1+i^2, i+i^3 \}. $$
Note that these are translates of the characteristic cyclotomic equation that defines $i$, as is the case with the elements of $B_1$ in general.

Then we have
$$ A=\{1+i^2, \omega+\omega\cdot i^2, \omega^2+\omega^2\cdot i^2, i+i^3, \omega\cdot i+\omega\cdot i^3,\omega^2\cdot i+\omega^2\cdot i^3  \}.  $$
We also have a choice to make for what we labeled $B'$ an integral basis of $\Z[i]$, whose elements are roots of unity. We can take for instance $B'=\{ 1,i^3 \}$. We find that $B''$ consists of just one relation, namely $B''=\{  1+\omega+\omega^2 \}$. We get
$$ A'=\{ 1+\omega+\omega^2, i^3+i^3\cdot \omega+i^3\cdot\omega^2  \},  $$
so that finally we have the basis $B_2=A\cup A'$.
\end{example}



Having described a basis for $R$, we now conclude this section by proving Lemma 2. We associate with vectors in $R$ or $M$, a length. The length in question is the one inherited from the embeddings $R\hookrightarrow M\hookrightarrow M\otimes \mathbb{R}$, where $\forall l$, $\mathbb{R}^l$ in all our treatment, is endowed with the Euclidean norm.

In the example we gave earlier for $M=\oplus_{i=0}^{11}\zeta_{12}^i$, the vector $-1+\zeta_n+2\zeta_{12}^4$ is of length $\sqrt{(-1)^2+1^2+2^2}=\sqrt{6}$. To compare with the indexing of the basis vectors of $M$, $\zeta_{12}^l$, as products $i^a\omega^b$, we have that every vector in $A$ will have length $\sqrt{2}$, while every vector in $A'$ will have length $\sqrt{3}$.

To practice our notation and conventions, the reader can verify for instance that all the elements of $A$ are pairwise orthogonal, being supported on disjoint coordinate vectors $i^a\omega^b$, and the same applies for the vectors of $A'$. We have on the other hand that for the first two vectors in the lists of $A$ and $A'$
$$ \langle 1+i^2,1+\omega+\omega^2\rangle=1.   $$
\begin{lemma}\label{length} Let $R$, the space of relation, be as above, and for its basis $B_j$, consider the fundamental mesh
$$ F=\left\{ \sum_{i=1}^{n-\phi(n)} x_i b_i| [x_1,\dots,x_{n-\phi(n)}]\in [0,1]^{n-\phi(n)} \right\}  ,$$  
where the $b_i$'s range over the above described elements of $B_j$. Then we have the maximum length among elements of $F$ at most
$$ j \sqrt{n} .$$
Again, here we denote by $j=\omega(n)$ the number of distinct prime divisors of $n$.

\end{lemma} 
\begin{proof}


We first observe that the chosen basis vectors have non-negative inner product. We can then infer that the maximum length in $F$, is achieved by the vector
$$  v_j=\sum_{a\in A}a+\sum_{a'\in A'}a'.  $$
It is also the case that the coordinates of the vectors $a\in A$ and $a'\in A'$ are either 0 or 1. This follows immediately from induction. What also easily follows from induction, is that each coordinate $\vctr{a}=[a_1,\dots,a_j]$ can be non-zero for at most $j$ vectors in $A$ or $A'$. Indeed, each coordinate $\vctr{a}$ is encountered in at most $j-1$ vectors of $A$ by the inductive hypothesis, while in at most one of the vectors in $A'$. The conclusion then follows.
\end{proof}
\section{ Divisibility by $\Phi_n(x)$, $n\leq k/\exp(\sqrt{\log k})$}
In section 5-7 we bound $\Po(\Phi_n(x)\mid F(x))$ in different regimes. We start here with the case $n\leq k/\exp(\sqrt{\log k})$. Like we explained in the first section of preliminaries, $\Phi_n\mid F(x)$ iff $f(\zeta_n)=0,$ where $f(x)=F(x)\mod (x^n-1)$, and $\zeta_n$ is a primitive $n$-th root of unity. Here, we use the lattice description of the relations extensively.

We start with the simple observation that for fixed $n$, the maximum weight that $f(x)$ can be assigned in the multinomial distribution is
$$ \binom{k}{k/n,\dots,k/n}/n^k,  $$
where in the case where $n\nmid k$, we interpret the above via the Gamma function. 
For instance, in the case $n=2$, we want to impose the necessary condition $2\mid k+1$. We have that $\Phi_2(x)=x+1$, and we want $f(x)$ to have $f(-1)=0$, which happens only when $f(x)=(k+1)/2\cdot x+(k+1)/2.$ This implies that for big $k$, we have
$$\Po(x+1\mid F(x))\sim \sqrt{\frac{2}{\pi k}}. $$
For the remaining $n$, we further split $\Phi_n(x)\mid F(x)$ into two events. Recall the space of relations in the roots of unity $R$ from the previous section. Let $\vctr{\mu}=(k/n,\dots,k/n)\in \Q^n.$ Call $B$ the ball $B=\{x\in R||x-\vctr{\mu}|\leq \sqrt{k}\log k\},$ where the distance is taken to be the Euclidean distance in $\Z^n$. We then consider the cases $f(x)\in R\cap B$, and $f(x)\in R\setminus B$.

Starting with the easiest of the two, we first bound $\Po(f(x)\in R\setminus B)$ with a concentration bound. From $|x-\vctr{\mu}|> \sqrt{k}\log k$, we have that at least one of the coordinates of $x$, say $x_i$, must satisfy $|x_i-k/n|> \sqrt{k/n}\log k$. Each individual coordinate of $x$ however, follows the binomial distribution with $k$ trials, and probability of success $1/n$. We then see from the Chernoff bound \ref{chernoff} that this happens with probability less than
$$ 2\exp(-\log^2 k/3). $$
Since this can happen for any of the $n$ coordinates of $x$, and summing over all $n$ in the range $3\leq n\leq k/\exp(\sqrt{\log k})$, we get that the probability that $f(x)\in R\setminus B$ does not exceed
$$ 2k^2 \exp(-\log^2 k/3), $$
which as $k$ grows tends to zero.

Now, for fixed $n$, we bound $\Po(f(x)\in R\cap B).$ Assume that $R\cap B\neq\emptyset,$ and fix some $x_0\in R\cap B$. Then for any other point $x\in R\cap B,$ we have that $|x-x_0|\leq 2\sqrt{k}\log k$, and $x-x_0\in R,$ lies in a $n-\phi(n)$ dimensional $\Z$-space. We are then interested in counting lattice points in a sphere of radius $2\sqrt{k}\log k$, and dimension $n-\phi(n)$. Using the basis described in the Preliminaries II section, we proceed by adjoining an oriented copy of the fundamental mesh $F$ at each lattice point. The number of lattice points in question then can be bounded by the volume of the resulting shape over the volume of $F$. Using Lemma \ref{length}, we see, in combination with the elementary fact that $\omega(n)=O(\log n/\log \log n)$, that the maximum vector length inside the fundamental mesh is $o(\sqrt{k})$, so that the resulting shape from adjoining the fundamental mesh is fully contained in a ball of radius $2.1\sqrt{k}\log k$ for big $k$. Using the formula for the volume of the $n-\phi(n)$-th dimensional sphere, we finally get that 
$$   \#(R\cap B)\leq (2.1\sqrt{k}\log k )^{n-\phi(n))}\frac{\pi^{\frac{n-\phi(n)}{2}}}{\Gamma(\frac{n-\phi(n)}{2}+1)} /\text{vol}(F).  $$
We also have that $\text{vol}(F)\geq 1$. That comes from the fact that for any lattice with basis $\{v_1,\dots, v_i\}$, we have that the volume of the fundamental mesh is given by
$$ \text{vol}(F)=\left|\det (\langle v_i,v_j\rangle)\right|^{1/2},  $$
which in our case is the square root of a natural number. We note here that it would be interesting to know whether $\text{vol}(F)$ has a nice description in terms of other well known invariants of $\Q(\zeta_n)$.

Putting everything together, we see that 
\begin{equation}\label{lattice} \Po(f(x)\in R\cap B)\leq   \binom{k}{k/n,\dots,k/n}/n^k \cdot (2.1\sqrt{k}\log k )^{n-\phi(n))}\frac{\pi^{\frac{n-\phi(n)}{2}}}{\Gamma(\frac{n-\phi(n)}{2}+1)} .\end{equation}
We now proceed to show that adding the above for $3\leq n \leq k/\exp(\sqrt{\log k})$, gives a sum that converges to zero as $k\rightarrow \infty$. Using the bound
\begin{equation}\label{Stirling}  \sqrt{2\pi n}\left(\frac{n}{e}\right)^n \leq n!\leq \sqrt{2\pi n}\left(\frac{n}{e}\right)^ne^{ \frac{1}{12n}} ,\end{equation}
from page 24 of \cite{Artin}, one gets that
\begin{align*}  \binom{k}{k/n,\dots,k/n}/n^k&=\frac{k!}{(k/n)!^n} \\
&<2\frac{1}{\sqrt{2\pi}^{n-1}}\frac{n^{n/2}}{k^{(n-1)/2}}.
 \end{align*}
Treating first the case where $n<\log^{2/3}k$, we then get that the sum in question is less than
\begin{align*}  &\sum_{3\leq n<\log^{2/3} k}2\frac{1}{\sqrt{2\pi}^{n-1}}\frac{n^{n/2}}{k^{(n-1)/2}} \cdot (2.1\sqrt{k}\log k )^{n-\phi(n))}\frac{\pi^{\frac{n-\phi(n)}{2}}}{\Gamma(\frac{n-\phi(n)}{2}+1)}\\
&\leq C \sum_{3\leq n<\log^{2/3} k} 2\sqrt{k}\frac{2.1^{n-\phi(n)}}{\sqrt{2\pi}^{n-1}}\cdot n^{n/2}\cdot \frac{1}{k^{\phi(n)/2}}\cdot \log^n k\\
&\overset{2.1/\sqrt{2\pi}<1,\phi(n)\geq 2,n<\log^{2/3}k}{\leq} C2\sqrt{k}\sum_{3\leq n<\log^{2/3} k} (\log^{2/3}k)^{0.5\cdot\log^{2/3} k}\cdot \frac{1}{k}\cdot (\log k)^{\log^{2/3} k}\\
&=\frac{2C}{k^{1/2}}\cdot  \log^{2/3} k\cdot  (\log^{2/3}k)^{0.5\cdot\log^{2/3} k}\cdot (\log k)^{\log^{2/3} k}\rightarrow 0,
 \end{align*}
where in the first inequality we used the boundedness of $\pi^{(n-\phi(n))/2}/\Gamma((n-\phi(n))/2+1)<C$. For the range $\log^{2/3} k\leq n\leq k/\exp(\sqrt{\log k})$, we also use the estimate
$$ \frac{1}{ \Gamma(\frac{n-\phi(n)}{2}+1)}<\sqrt{2\pi}\frac{e^{\frac{n-\phi(n)}{2}}}{\left( \frac{n-\phi(n)}{2}  \right)^{\frac{n-\phi(n)}{2}}}.  $$
We then find
\begin{align*}
&\sum_{\log^{2/3} k\leq n\leq k/\exp(\sqrt{\log k})}2\frac{1}{\sqrt{2\pi}^{n-1}}\frac{n^{n/2}}{k^{(n-1)/2}} \cdot (2.1\sqrt{k}\log k )^{n-\phi(n))}\frac{\pi^{\frac{n-\phi(n)}{2}}}{\Gamma(\frac{n-\phi(n)}{2}+1)}\\
&\leq 2\sqrt{k}\sum_{\log^{2/3} k\leq n\leq k/\exp(\sqrt{\log k})}D^n\frac{n^{n/2}}{(n-\phi(n))^{(n-\phi(n))/2}}\cdot \frac{k^{(n-\phi(n))/2}}{k^{n/2}}\cdot \log^n k.
\end{align*}
Using the monotonicity of $g(x)=k^x/x^x$, for $x\in[1,n/2]$, we see that if $x=n-\phi(n)<n/2$, we get an upper bound in the above by replacing $n-\phi(n)$ with $n/2$. This gives terms of the form $D^n(n/k)^{n/4}\log^n k$. If instead we have $n-\phi(n)\geq n/2$, we get then the bound
$$2\sqrt{k}\sum_{\log^{2/3} k\leq n\leq k/\exp(\sqrt{\log k})}E^n\left(\frac{n}{k}\right)^{\phi(n)/2}\cdot \log^n k.  $$
Both bounds can be replaced with the weaker
$$  2\sqrt{k}\sum_{\log^{2/3} k\leq n\leq k/\exp(\sqrt{\log k})}F^n\left(\frac{n}{k}\right)^{\phi(n)/4}\cdot \log^n k.    $$
Using the fact that for big $n$ as in this case $\phi(n)>n/(2\log \log n)$, see \cite{facts} Theorem 15, we see that 
\begin{align*}
\left(\frac{n}{k}\right)^{\phi(n)/4}&\leq \exp(\log(n/k)n/(8\log\log n)  )\\
&\leq \exp(\log(n/k)n/(8\log\log k)  )\\
&\leq \exp(-\sqrt{\log k}\cdot n/(8\log\log k) )\\
&<\exp(-\log^{1/3} k)^n.
\end{align*}
This implies after including the powers $F^n\cdot \log^n k$,  that the sum in question is smaller than the sum of a geometric series with the same initial term and ratio, say $\exp(-\log^{1/4} k)$. The initial term would be $F^n[\frac{n}{k}]^n\log^n k$ evaluated at $n=\log^{2/3}k$. This then concludes that the sum is in fact very small and approaches zero as $k\rightarrow \infty$.



In this section then, we have concluded that $\Po(\lor_n (\Phi_n\mid F(x)))$, when $1\leq n\leq k/\exp(\sqrt{\log k})$, has an upper bound that tends to zero as $k\rightarrow \infty.$ We remark that $n=1$ is trivial, since $x-1$ never divides $F(x)$, and that in our arguments we used in the notation of Preliminaries I, the vector $\vctr{c}$, instead of the typically correct $\vctr{c'}$. The reader can be convinced that this does not affect the validity of the argument, and the choice was made to simplify the presentation.

\section{Divisibility by $\Phi_n(x)$ for $k/\exp(\sqrt{\log k})<n\leq e^{k/(24\log k)}$}
We start with the observation that $\phi(n)>n/(2\log \log n)$ implies for the above range that $\phi(n)>n/(2\log k)$. We want to make use of the fact that the condition $\Phi_n(x)\mid F(x)$, or equivalently $\Phi_n(x)\mid f(x)$, implies that the initial $\phi(n)$ coefficients of $f(x)$ in degree 0 up to $\phi(n)-1$, are determined by the remaining higher degree coefficients by precisely the condition $ \Phi_n(x)\mid f(x)$, by the division algorithm. Call $K_n$ the sum of the coefficients of $f(x)$ in that initial range. We also observe that the quantity $K_n$, as $F(x)$ ranges, is a random variable following $B(k,\phi(n)/n)$, the binomial distribution with $k$ trials, and probability of success $\phi(n)/n$. Then the probability that $\Phi_n(x)\mid F(x)$, conditional on $f(x)$ having $K_n=m,$ does not exceed the weight of the biggest atom in the multinomial distribution with $m$ trials and $\phi(n)$ outcomes.

We proceed then with two steps. One is to show that with high probability for each $n$ we have $K_n>k\phi(n)/(2n)$. With that being the case, we can then bound the probability that $\Phi_n(x)\mid F(x)$, by finding a uniform bound for the heaviest atoms in the sequence of multinomial distributions $M(m,\phi(n)),$ with $m$ trials and $\phi(n)$ equally likely outcomes, where $m$ ranges in $[k\phi(n)/(2n),k]$. When this is complete, we add the bounds for all $n$ in the given range.

Using that $K_n$ follows $B(k,\phi(n)/n)$, we get by the Chernoff bound that
$$ \Po(K_n<k\phi(n)/(2n))\leq 2\exp(- k\phi(n)/(12n)) \leq2 \exp(-k/(12\log k)). $$
Adding over $k/\exp(\sqrt{\log k})<n\leq e^{k/(24\log k)}$, we see that 
$$\Po(\lor(K_n<k\phi(n)/(2n)))\leq2 \exp(- k/(24\log k)).$$
To bound the atoms of the multinomial distribution, it is convenient to split into two ranges, one where $m\geq\phi(n)$, and one where $\phi(n)>m$. When $m\geq \phi(n)$, we have that the heaviest atom has weight at most
$$  \frac{1}{\phi(n)^m}\frac{m!}{\Gamma(m/\phi(n)+1)^{\phi(n)} } . $$
Similar to $\ref{Stirling}$, we have
$$  \sqrt{2\pi x}\cdot x^{x}\cdot e^{-x}<\Gamma(x+1),  $$
and both can be found in \cite{Artin}. From that we get an upper bound of
$$  \frac{\sqrt{m}}{\sqrt{2\pi}^{\phi(n)-1}} \left(\frac{\phi(n)}{m} \right)^{\phi(n)}\leq\frac{\sqrt{m}}{\sqrt{2\pi}^{\phi(n)-1}}\leq\frac{\sqrt{k}}{\sqrt{2\pi}^{\phi(n)-1}}  $$
for the weight of the atoms with $m\geq \phi(n)$. We observe that the inequality $k\geq m\geq \phi(n),$ gives for example $n<k\log k$, and adding the atom weights in this range gives a sum that approaches zero as $k\rightarrow \infty$. In the remaining $\phi(n)>m$ range, we have that the maximum weight of each atom is
$$   \frac{m!}{\phi(n)^m},$$
where we apply the much simpler bound
\begin{align*}  \frac{m!}{\phi(n)^m}&\leq m/e^{m-1}\cdot m^m/\phi(n)^m\\
&\leq m/e^{m-1}\leq e k\phi(n)/(2n)/e^{k\phi(n)/(2n)} \\
&\leq \frac{e}{2}k /e^{ (k/(4\log k))}.
\end{align*}
This is more than enough to cover the range $k/\exp(\sqrt{\log k})<n\leq e^{k/(24\log k)}$.

\section{$e^{k/(24\log k)}\leq n\leq N\log N$}
In this final section we treat the condition $\Phi_n(x)\mid F(x)$ for big $n$. This is where we invoke Corollary 1, according to which and together with Lemma 3, we can restrict our attention to those $n$ that have squarefree part of order less than $ e^{c\sqrt{k}\log k} $ for some constant $c$. 

Let $n=ab$, where $a$ is the squarefree part of $n$. Invoking then the Conway and Jones Proposition \ref{Conway} from Preliminaries II, we have that $F(\zeta_n)=0$, with $F(x)=1+\sum_{i=1}^k x^{n_i}$ implies that 
$$S_j= \sum_ {n_i\equiv j \mod b}\zeta_n^{n_i}=0,\quad\forall j.   $$
We will rely solely on the fact we do not expect many exponents to lie in any residue class $\mod b$. We first of all observe that $S_j$ cannot be zero if there is only a single exponent $n_i\equiv j \mod b$. We will as a first step show that it is with low probability that we expect for any $j$ to have three or more exponents $n_i\equiv j \mod b$. This will then allow us to infer that outside of an event of small probability, to require $S_j=0$, makes necessary that for all $j$, $S_j$  if non-empty, consists of exactly two summands. We conclude by bounding the probability that $S_j$ has always exactly two summands.

Call $A_j=\{i: 1\leq i\leq N,\,n\equiv j \mod b\}.$ Then we have $\forall j$, that $\# A_j=N/b+\epsilon_j$, where $\epsilon_j\in (-1,1)$. Following the same treatment as in the Preliminaries I, we see that if we fix $j$, having three elements in $A_j$, can be achieved in
$$   \binom{\#A_j}{3}\cdot \binom{N-3}{k-3}  $$ 
ways. This gives a probability to having three exponents $n_i\in A_j$, of
\begin{align*}  \frac{ \binom{\#A_j}{3}\cdot \binom{N-3}{k-3}}{\binom{N}{k}}&= \frac{1}{6}\frac{(\frac{N}{b}+\epsilon_j)(\frac{N}{b}+\epsilon_j-1)(\frac{N}{b}+\epsilon_j-2)}{\frac{N(N-1)(N-2)}{k(k-1)(k-2)}}\\
&\ll\frac{ \frac{N^3}{b^3} }{\frac{N(N-1)(N-2)}{k(k-1)(k-2)}}\\
&=O\left(\frac{k(k-1)(k-2)  }{3b^3}\right).
\end{align*}
Taking into account the fact that the above can happen for $b$ choices of $j$, we have that the probability of having at least three exponents in any $A_j$, is at most
$$ O\left( \frac{k^3}{b^2}\right).  $$
Using the bound on the squarefree part of $n$, we have that $b\geq n/e^{c\sqrt{k}\log k}$, so that in total, the probability of having a three term $A_j$ for any $n>e^{k/(24\log k)}$ is at most
$$  \sum_{ n>e^{k/(24\log k)}}\frac{k^3 e^{2c\sqrt{k}\log k}}{n^2}, $$
which vanishes as $k$ grows.

Having established this claim, we now bound the probability that all $A_j$ have two exponents. Again we need to remark that having normalized $F(x)$ with constant term 1, makes automatically $A_0$ participate non-trivially with at least one summand in $S_0$. The argument we run earlier is not sensitive to small changes in the number of terms in $S_0$. If we want to be exact, we have to consider two cases in what follows, one where $S_0$ has two or three terms including one. Let's instead for notational unity just address the case of having two terms in each non-trivial $S_j$, with $k$ terms in total, disregarding the special form of $S_0$. The reader can be convinced that this is not to the detriment of the argument.

For a fixed choice of the $k/2$ $j$'s that are non-trivial, we have 
$$  \prod_j \binom{\#A_j}{2}   $$
ways to choose two exponents from each $A_j$, therefore gving a probability of this happening, of no more than
\begin{align*}
\frac{\prod_j \binom{\#A_j}{2}}{\binom{N}{k}}&=\frac{\prod_j \frac{1}{2}(\frac{N}{b}+\epsilon_j)(\frac{N}{b}+\epsilon_j-1) }{\binom{N}{k}}\\
&\leq k!\frac{\prod_j \frac{1}{2}(\frac{N}{b}+\epsilon_j)(\frac{N}{b}+\epsilon_j-1)}{(N+1-k)^k}\\
&\overset{\frac{N}{b}\geq 1,\epsilon_j<1}{\leq}k!\frac{\prod_j (\frac{N}{b})^2}{(N+1-k)^k}\\
&\leq \frac{ k!1.01^k }{b^k},
\end{align*}
where in the last inequality we use that $N$ is a lot bigger than $k$. Considering all the $\binom{b}{k/2}$ ways that we can pick such $j\mod b$, it is clear that having all $A_j$'s capacitate exactly two exponents or none, happens with probability at most
$$  \frac{k!1.01^k}{(k/2)!}\cdot \frac{1}{b^{k/2}}.  $$
Same as above, using $b\geq n/e^{c\sqrt{k}\log k}$, we conclude that the total of these contributions is negligable as $k$ grows.
\section{Acknowlegments}
The author is indebted to Michael Filaseta for suggesting the topic, and providing many key references, most notably \cite{schinzel}.

\end{document}